\documentclass[12pt,a4paper]{amsart}
\usepackage{mathrsfs}
\usepackage{amssymb,amsmath,amsthm,color}
\usepackage{graphicx,mcite}
\usepackage{hyperref}
\usepackage{url}
\usepackage{setspace}
\usepackage{enumerate}
\newtheorem{theorem}{Theorem}[section]
\newtheorem{lemma}[theorem]{Lemma}
\newtheorem{proof of lemma}[theorem]{Proof of Lemma}
\newtheorem{proposition}[theorem]{Proposition}

\theoremstyle{definition}
\newtheorem{definition}[theorem]{Definition}

\newtheorem{remark}[theorem]{Remark}

\numberwithin{equation}{section}

\begin{document}
\title[Benedicks-Amrein-Berthier theorem]
{Benedicks-Amrein-Berthier theorem for the Heisenberg motion group and quaternion Heisenberg group}

\author{Somnath Ghosh and R.K. Srivastava}

\address{Department of Mathematics, Indian Institute of Technology, Guwahati, India 781039.}
\email{gsomnath@iitg.ac.in, rksri@iitg.ac.in}


\date{\today}


\begin{abstract}
Since $(\mathbb{H}^n\rtimes U(n),U(n))$ is a Gelfand pair, an exact analogue of the
Heisenberg group result due to Narayanan and Ratnakumar is not possible for the
Heisenberg motion group. In this article, we prove that if the Weyl transform of a
finitely supported integrable function on the Heisenberg motion group is non-zero
only for finitely many Fourier-Wigner pieces and have finite rank, then the function
must be zero. We also prove an analogue of the Heisenberg group result on the quaternion
Heisenberg group. In the end, a quantitative interpretation of these results is described
through strong annihilating pair for the Weyl transform.
\end{abstract}

\maketitle

\section{Introduction}\label{section1}
In an interesting article \cite{B}, Benedicks proved that if $f\in L^1(\mathbb{R}^n),$
then both the sets $\{x\in \mathbb{R}^n:f(x)\neq0\}$ and $\{\xi\in \mathbb{R}^n:\hat{f}(\xi)\neq0\}$
cannot have finite Lebesgue measure, unless $f=0.$ Concurrently, in the article \cite{AB}, Amrein-Berthier
reached the same conclusion via the Hilbert space theory. The aforesaid fundamental result got further
attention in general Lie groups.

\smallskip

Let $G$ be a locally compact group and $\hat{m}$ denotes the Plancherel measure on the unitary
dual group $\hat{G}.$ Then $G$ is said to satisfy qualitative uncertainty principle (QUP) if
for each $f \in L^2(G)$ with $m\{x\in G: f(x)\neq 0\}<m(G)$ and
\begin{equation}\label{exp79}
\int_{\hat{G}}\text{rank}\hat f(\lambda)\,d\hat{m}(\lambda)<\infty
\end{equation}
implies $f=0.$ In \cite{AL}, QUP was proved for certain unimodular groups of type I. A brief survey
of QUP is presented in \cite{FS}. In the case of the Heisenberg group $\mathbb H^n,$ the condition
(\ref{exp79}) of QUP implies $\hat f$ should be supported on a set of finite Plancherel measure together
with $\text{rank}\hat f(\lambda)$ is finite for almost all $\lambda.$

\smallskip

In \cite{NR}, Narayanan and Ratnakumar proved that if $f\in L^1(\mathbb{H}^n)$ is
supported on $B\times \mathbb{R},$ where $B$ is a compact subset of $\mathbb{C}^n$,
and $\hat{f}(\lambda)$ has finite rank for each $\lambda,$ then $f=0.$ Thereafter,
Vemuri \cite{V} replaced the compactness condition on $B$ by finite measure.
In \cite{CGS}, authors consider $B$ as a rectangle in $\mathbb R^{2n}$ while proving
an analogous result on step two nilpotent Lie groups and a version of this result, with
a strong assumption on rank, derived therein for the Heisenberg motion group. Later in
the article \cite{GS}, this result is extended to arbitrary set $B$ of finite measure
for general step two nilpotent Lie groups. In this article, we prove the result on the
quaternion Heisenberg group when $B$ is an arbitrary set of finite measure and the
following analogue result on the Heisenberg motion group.

\smallskip

First, consider the Heisenberg motion group $G,$ which is the semidirect product of
$\mathbb H^n$ and $U(n),$ the unitary group on $\mathbb C^n.$ Denote $K=U(n).$ Then,
due to the fact that $(G, K)$ is a Gelfand pair \cite{BJR}, the Fourier transform of
a $K$-bi invariant integrable function has rank one, irrespective of support of the
function. Thus, an exact analogue of the Heisenberg group result due to Narayanan and
Ratnakumar \cite{NR} is not inevitable for the Heisenberg motion group.

\smallskip

However, we prove that if the Weyl transform of a finitely supported integrable function
is non-zero only for finitely many Fourier-Wigner pieces and have finite rank, then
the function is zero. Consequently, we obtain that if each Fourier-Wigner piece of a
non-trivial function has finite support, then all of its Fourier transform can not have
finite rank. A quantitative interpretation of this result is described through strong
annihilating pair for the Weyl transform.

\smallskip

The proof of our results proceeds through the Hilbert space theory. However, specifying
the appropriate set of projections in the setups of the Heisenberg motion group and
quaternion Heisenberg group was a major bottleneck. This result, as of now, is the most
general analogue of Benedicks-Amrein-Berthier theorem in these setups.

\section{Heisenberg motion group}
The Heisenberg group $\mathbb H^n=\mathbb C^n\times\mathbb R$ is a step two
nilpotent Lie group having center $\mathbb R$ that equipped with the group law
\[(z,t)\cdot(w,s)=\left(z+w,t+s+\frac{1}{2}\text{Im}(z\cdot\bar w)\right).\]

\smallskip

By Stone-von Neumann theorem, the infinite dimensional irreducible unitary
representations of $\mathbb H^n$ can be parameterized by $\mathbb R^\ast=\mathbb
R\smallsetminus\{0\}.$ That is, each $\lambda\in\mathbb R^\ast$ defines a
Schr\"{o}dinger representation $\pi_\lambda$ of $\mathbb H^n$ via
\[\pi_\lambda(z,t)\varphi(\xi)=e^{i\lambda t}e^{i\lambda(x\cdot\xi+\frac{1}{2}x\cdot y)}\varphi(\xi+y),\]
where $z=x+iy$ and $\varphi\in L^2(\mathbb{R}^n).$

\smallskip

Having chosen sublaplacian $\mathcal L$ of the Heisenberg group $\mathbb H^n$ and its
geometry, there is a larger group of isometries that commute with $\mathcal L$, known as
Heisenberg motion group. The Heisenberg motion group $G$ is the semidirect product of
$\mathbb H^n$ with the unitary group $K=U(n).$ Since $K$ defines a group
of automorphisms on $\mathbb H^n,$ via $k\cdot(z,t)=(kz,t),$ the group law on
$G$ can be expressed as
\[(z,t,k_1)\cdot(w,s,k_2)=\left(z+k_1w, t+s-\frac{1}{2}\text{Im} (k_1w\cdot\bar z),k_1k_2\right).\]
Since a right $K$-invariant function on $G$ can be thought of as a function on $\mathbb H^n,$
the Haar measure on $G$ is given by $dg=dzdtdk,$ where $dzdt$ and $dk$ are the normalized
Haar measure on $\mathbb H^n$ and $K,$ respectively.

\smallskip

For $k\in K$ define another set of representations of the Heisenberg group
$\mathbb H^n$ by $\pi_{\lambda,k}(z,t)=\pi_\lambda(kz,t).$ Since $\pi_{\lambda,k}$
agrees with $\pi_\lambda$ on the center of $\mathbb H^n,$  it follows by
Stone-Von Neumann theorem for the Schr\"{o}dinger representation that
$\pi_{\lambda,k}$ is equivalent to $\pi_\lambda.$ Hence there exists an
intertwining operator $\mu_\lambda(k)$ satisfying
\[\pi_\lambda(kz,t)=\mu_\lambda(k)\pi_\lambda(z,t)\mu_\lambda(k)^\ast.\]
Then $\mu_\lambda$ can be thought of as a unitary
representation of $K$ on $L^2(\mathbb R^n),$ called
metaplectic representation. For details, we refer to \cite{BJR}.
Let $(\sigma,\mathcal H_\sigma)$ be an irreducible unitary representation of $K$
and $\mathcal H_\sigma=\text{span}\{e_j^\sigma:1\leq j\leq d_\sigma\}.$ For
$k\in K,$ the matrix coefficients of the representation $\sigma\in\hat K$ are
given by \[\varphi_{ij}^\sigma(k)=\langle\sigma(k)e_j^\sigma, e_i^\sigma\rangle,\]
where $i, j=1,\ldots,  d_\sigma.$ By Peter-Weyl theorem for compact groups \cite{Su},
it follows that the set  $\{\sqrt{d_\sigma}\varphi_{ij}^\sigma:1\leq i,j\leq d_\sigma,
\sigma\in\hat K\}$ is an orthonormal basis for $L^2(K).$

\smallskip

Let $\phi_\alpha^\lambda(x)=|\lambda|^{\frac{n}{4}}\phi_\alpha(\sqrt{|\lambda|}x);~\alpha\in\mathbb Z_+^n,$
where $\phi_\alpha$'s are the Hermite functions on $\mathbb R^n.$ Since for each $\lambda\in\mathbb R^\ast,$
the set $\{\phi_\alpha^\lambda : \alpha\in\mathbb Z_+^n \}$ forms an orthonormal basis for $L^2(\mathbb R^n),$
letting $P_m^\lambda=\text{span}\{\phi_\alpha^\lambda:~ |\alpha|=m\},$ $\mu_\lambda$ becomes an irreducible
unitary representation of $K$ on $P_m^\lambda.$ Hence, the action of $\mu_\lambda$ can be realized on
$P_m^\lambda$ by
\begin{equation}\label{exp50}
\mu_\lambda(k)\phi_\gamma^\lambda=\sum_{|\alpha|=|\gamma|}\eta_{\alpha\gamma}^\lambda(k)\phi_\alpha^\lambda,
\end{equation}
where $\eta_{\alpha\gamma}^\lambda$'s are the matrix coefficients of $\mu_\lambda(k).$
Define a bilinear form $\phi_{\alpha}^\lambda\otimes e_j^\sigma$ on
$L^2(\mathbb R^n)\times\mathcal H_\sigma$ by
$\phi_{\alpha}^\lambda\otimes e_j^\sigma=\phi_{\alpha}^\lambda e_j^\sigma.$ Then
$\{\phi_{\alpha}^\lambda\otimes e_j^\sigma: \alpha\in\mathbb Z_+^n,1\leq j\leq d_\sigma\}$
forms an orthonormal basis for $L^2(\mathbb R^n)\otimes\mathcal{H}_\sigma.$
Denote  $\mathcal H_\sigma^2=L^2(\mathbb R^n)\otimes\mathcal{H}_\sigma.$
Define a representation $\rho_\sigma^\lambda$
of $G$ on the space $\mathcal H_\sigma^2$ by
\[\rho_\sigma^\lambda(z,t,k)=\pi_\lambda(z,t)\mu_\lambda(k)\otimes\sigma(k).\]
Then $\rho_\sigma^\lambda$ are all possible irreducible unitary representations of $G$
those participate in the Plancherel formula \cite{S}. Thus, in view of the above discussion,
we shall denote the partial dual of the group $G$ by $ G'\cong\mathbb R^\ast\times\hat K.$
For $(\lambda,\sigma)\in G',$ the Fourier transform of $f\in L^1(G),$ defined by
\[\hat f(\lambda,\sigma) =\int_K\int_\mathbb{R}\int_{\mathbb{C}^n} f(z,t,k)\rho_\sigma^\lambda(z,t,k)dzdtdk,\]
is a bounded linear operator on $\mathcal H_\sigma^2.$ As the Plancherel formula \cite{S}
\[\int_K\int_{\mathbb{H}^n} |f(z,t,k)|^2 dzdtdk=(2\pi)^{-n}\sum_{\sigma\in\hat K}d_\sigma
\int_{\mathbb{R}\setminus\{0\}}\|\hat f(\lambda,\sigma)\|^2_{HS}|\lambda|^nd\lambda\]
holds for $f\in L^2(G),$ it follows that $\hat f(\lambda,\sigma)$ is a Hilbert-Schmidt
operator on $\mathcal H_\sigma^2.$

\smallskip

In order to prove our main result on the Heisenberg motion group $G,$ it is
enough to consider a similar proposition for the Weyl transform on
$G^\times=\mathbb C^n\times K.$ For that, we require to set some preliminaries
about the Weyl transform on $G^\times.$

\smallskip

Let $f^\lambda$ be the inverse Fourier transform of the function $f$
in the $t$ variable defined by
\begin{align}\label{exp57}
f^\lambda(z,k)=\int_{\mathbb R}f(z,t, k)e^{i\lambda t}dt.
\end{align}
Then
\[\hat f(\lambda,\sigma)=\int_K\int_{\mathbb{C}^n} f^\lambda(z,k)\rho_\sigma^\lambda(z,k)dzdk,\]
where $\rho_\sigma^\lambda(z,k)=\rho_\sigma^\lambda(z,0,k).$

\smallskip

For $(\lambda, \sigma)\in G',$ define the Weyl transform $W_\sigma^\lambda$ on $L^1(G^\times)$ by
\[W_\sigma^\lambda(g)=\int_{K}\int_{\mathbb{C}^n}g(z,k)\rho_\sigma^\lambda(z,k)dzdk.\]
Then $\hat{f}(\lambda,\sigma)=W_{\sigma}^\lambda (f^\lambda),$ and hence
$W_\sigma^\lambda(g)$ is a bounded operator if $g\in L^1(G^\times).$ On the other hand,
if $g\in L^2(G^\times)$ then $W_\sigma^\lambda(g)$ becomes a Hilbert-Schmidt operator
satisfying the Plancherel formula
\begin{align}\label{exp51}
\int_{K}\int_{\mathbb{C}^n}|g(z,k)|^2dzdk =(2\pi)^{-n}|\lambda|^n\sum_{\sigma\in\hat K}
d_\sigma\left|\left|W_\sigma^\lambda(g)\right|\right|_{HS}^2.
\end{align}

\smallskip

\noindent \textbf{Fourier-Wigner representation:}
Define the Fourier-Wigner transform $V_\zeta^\eta$ of the functions $\zeta,\eta\in \mathcal{H}_\sigma^2$ by
\[V_\zeta^\eta(z,k)=(2\pi)^{-\frac{n}{2}}|\lambda|^{\frac{n}{2}}\left\langle\rho_\sigma^\lambda(z,k)\zeta,\eta\right\rangle,\]
where $(z,k)\in G^\times.$ The following orthogonality relation is derived in \cite{CGS}.
A version of this result is also appeared in \cite{S}.
\begin{lemma}\label{lemma50}
For $\zeta_l, \eta_l\in \mathcal{H}_\sigma^2,~l=1,2,$ the corresponding Fourier-Wigner transforms
satisfy
\[\int_{K}\int_{\mathbb{C}^n}V_{\zeta_1}^{\eta_1}(z,k)\overline{V_{\zeta_2}^{\eta_2}(z,k)}dzdk
=\left\langle \zeta_1,\zeta_2 \right\rangle\overline{\left\langle \eta_1,\eta_2 \right\rangle}.\]
\end{lemma}
In particular, $V_\zeta^\eta\in L^2(G^\times).$ Let $V_\sigma^\lambda=\overline{\text{span}}
\{V_\zeta^\eta:\zeta,\eta\in \mathcal{H}_\sigma^2\}$ and set $\Psi_{\alpha,j}^\sigma=\phi_\alpha^\lambda\otimes e_j^\sigma.$
Since $B_\sigma^\lambda=\{\Psi_{\alpha,j}^\sigma~:\alpha\in\mathbb Z_+^n, 1\leq j \leq d_\sigma\}$ forms an
orthonormal basis for $\mathcal{H}_\sigma^2,$ by Lemma \ref{lemma50}, we infer that
\[V_{B_\sigma^\lambda}=\left\{V_{\Psi_{\alpha_1,j_1}^\sigma}^{\Psi_{\alpha_2,j_2}^\sigma}:~
\Psi_{\alpha_1,j_1}^\sigma,\Psi_{\alpha_2,j_2}^\sigma\in B_\sigma^\lambda\right\}\]
is an orthonormal basis for $V_\sigma^\lambda.$ The next result, which is followed as a corollary
of the Peter-Weyl theorem \cite{Su}, will be the desire decomposition of $L^2(G^\times).$
\begin{proposition}\label{prop51}
The set $\bigcup\limits_{\sigma\in\hat K}V_{B_\sigma^\lambda}$ is an orthonormal basis for $L^2(G^\times).$
\end{proposition}
Since $V_{B_\sigma^\lambda}$ is an orthonormal basis for $V_\sigma^\lambda,$ by
Proposition \ref{prop51}, we infer that $L^2(G^\times)=\bigoplus\limits_{\sigma\in\hat K}V_\sigma^\lambda.$
We shall call this as the \textbf{\emph{Fourier-Wigner decomposition}} and $V_\sigma^\lambda$ as
\textbf{\emph{Fourier-Wigner representation}} of $G^\times.$

\begin{remark}\label{rk1}
Though the decomposition in Proposition \ref{prop51} is being followed by Peter-Weyl
theorem, it is quite finer than the usual Peter-Weyl decomposition of function on $K,$
due to the presence of the metaplectic representation. And as an effect, even if
$g\in  L^2(G^\times)$ is $K$-bi-invariant on $G^\times,$ it need not fall into the
trivial Fourier-Wigner representation. This fact can be explained more explicitly
via the following example.
\end{remark}
Consider the one dimension Heisenberg motion group $\mathbb H^1\rtimes U(1).$
Realize $U(1)\cong S^1.$ Let $(z,t,e^{i\theta})\in \mathbb H^1\rtimes U(1).$ Then for each
$(\lambda,\alpha)\in\mathbb{R}^*\times\mathbb{Z},$ the unitary irreducible representations
$(\rho_\alpha^\lambda,L^2(\mathbb{R}))$ of $\mathbb H^1\rtimes U(1)$ can be defined by
\[\rho_\alpha^\lambda(z,t,e^{i\theta})=e^{-i\alpha\theta}\pi_\lambda(z,t)\mu_\lambda(e^{i\theta}).\]
In fact, the action of $\rho_\alpha^\lambda$ on Hermite function $\phi_\beta^\lambda,$
where $\beta\in \mathbb{Z}_+,$ will be given by
\begin{align}\label{exp1}
\rho_\alpha^\lambda(z,t,e^{i\theta})\phi_\beta^\lambda=e^{-i(\alpha-\beta)\theta}\pi_\lambda(z,t)\phi_\beta^\lambda.
\end{align}
For more details, see \cite{ST}. By Proposition \ref{prop51}, we have
$L^2(\mathbb C \times S^1)=\oplus_{\alpha\in \mathbb{Z}}V_\alpha,$ where
$$V_{\alpha'}=\{\varphi\in L^2(\mathbb C \times S^1):W_\alpha^\lambda(\varphi)=0 \text{ for all }\alpha\neq \alpha'\}.$$
From (\ref{exp1}) it follows that
\begin{align*}
\langle\rho_\alpha^\lambda(z,e^{i\theta})\phi_\beta^\lambda,\phi_\gamma^\lambda\rangle
&=e^{-i(\alpha-\beta)\theta}\langle\pi_\lambda(z)\phi_\beta^\lambda,\phi_\gamma^\lambda\rangle \\
&=e^{-i(\alpha-\beta)\theta}\Phi_{\beta,\gamma}^\lambda(z),
\end{align*}
where $\Phi_{\beta,\gamma}^\lambda$ is the special Hermite function.
Denote $\tilde{\Phi}_{\beta,\gamma}^{\alpha,\lambda}(z,e^{i\theta})=e^{-i(\alpha-\beta)\theta}\Phi_{\beta,\gamma}^\lambda(z).$
Then $\{\tilde{\Phi}_{\beta,\gamma}^{\alpha,\lambda} :\beta,\gamma\in \mathbb{Z}_+\}$ will be
an orthonormal basis for $V_\alpha^\lambda.$ In particular, corresponding to the trivial representation,
$\tilde{\Phi}_{\beta,\gamma}^{0,\lambda}(z,e^{i\theta})
=e^{i\beta\theta}\Phi_{\beta,\gamma}^\lambda(z);$ $\beta,\gamma\in \mathbb{Z}_+,$ are basis
elements of $V_0^\lambda.$ Thus, the presence of $e^{i\beta\theta}$ in the basis,
concludes that an arbitrary $h\in L^2(\mathbb{C})$ need not be contained in $V_0^\lambda.$
To be explicit, consider a finitely supported function $h\in L^2(\mathbb C).$
Further, define a function $g$ on $\mathbb C \times S^1$ by $g(z,e^{i\theta})=h(z).$
The Weyl transform of $g$ is defined by
\[W_\alpha^\lambda(g)=\int_{\mathbb C \times [0,2\pi]}g(z,e^{i\theta})\rho_\alpha^\lambda(z,e^{i\theta})dzd\theta.\]
From (\ref{exp1}), we have
\[W_\alpha^\lambda(g)\phi_\beta^\lambda=W_\lambda(h)\phi_\beta^\lambda\cdot \int_0^{2\pi}e^{-i(\alpha-\beta)\theta} d\theta.\]
Hence, for each $\alpha\in\mathbb Z,$ $W_\alpha^\lambda(g)$ can have rank at most one. Further
$W_\beta^\lambda(g)\phi_\beta^\lambda=W_\lambda(h)\phi_\beta^\lambda$ for
$\beta\in \mathbb{Z}_+.$ Since $h$ is a non-zero function with finite support,
$W_\lambda(h)$ cannot be a finite rank operator, see \cite{GS,NR}. Therefore, there exist
infinitely many $\beta$ such that $W_\beta^\lambda(g)\Phi_\beta^\lambda\neq 0.$
Hence, $g$ is not contained in any $V_\alpha^\lambda.$ In fact, $g$ fails to be a member of any finite
union of $V_\alpha^\lambda$'s.

\smallskip

The above discussion brings forward the following question. Does there exist a non-trivial
finitely supported function $g\in L^1(G^\times)$ whose Weyl transform $W_\sigma^\lambda(g)$
has finite rank only for finitely many $\sigma$ and $W_\sigma^\lambda(g)=0$ otherwise?
The non-existence of such a function is guaranteed by Proposition \ref{prop58}.

\smallskip

Next, we prove the inversion formula for the Weyl transform $W_{\sigma}^\lambda$
which is a key ingredient while proving our main result. For this, we need the
fact that
\begin{equation}\label{exp52}
\rho_{\sigma}^\lambda(z,k_1)\rho_{\sigma}^\lambda (w,k_2)=e^{-\frac{i\lambda}{2} \text{Im} (k_1 w.\bar{z})}
\rho_{\sigma}^\lambda (z+k_1 w ,k_1k_2),
\end{equation}
where $(z,k_1),(w,k_2)\in G^\times.$
\begin{theorem}$(\textbf{Inversion formula})$\label{th50}
Let $g\in L^1\cap L^2(G^\times).$ Then
\begin{equation}\label{exp55}
g(z,k)=(2\pi)^{-n}|\lambda|^n\sum\limits_{\sigma\in \hat{K}}d_\sigma \text{tr}
(W_{\sigma}^\lambda(g)(\rho_{\sigma}^\lambda)^*(z,k)),
\end{equation}
where the series converges in $L^2(G^\times).$
\end{theorem}

\begin{proof}
For $(z,k_1)\in G^\times,$ we have
\begin{align*}
W_{\sigma}^\lambda(g)(\rho_{\sigma}^\lambda)^*(z,k_1)&=\int_{G^\times} g(w,k_2)\rho_{\sigma}^\lambda(w,k_2)
\rho_{\sigma}^\lambda(-k_1^{-1}z,k_1^{-1}) dw dk_2 \\
&=\int_{G^\times} g(w,k_2)e^{\frac{i\lambda}{2} \text{Im} (k_2k_1^{-1}z.\bar{w})}
\rho_{\sigma}^\lambda (w-k_2k_1^{-1}z,k_2k_1^{-1}) dw dk_2.
\end{align*}
Hence $\text{tr}\left(W_{\sigma}^\lambda(g)(\rho_{\sigma}^\lambda)^*(z,k_1)\right)$ is equal to
\begin{align*}
\sum\limits_{\substack {\gamma\in \mathbb{N}^n\\ 1\leq j\leq d_{\sigma}}} \int_{G^\times}
g(w,k_2)e^{\frac{i\lambda}{2} \text{Im} (k_2k_1^{-1}z.\bar{w})}
\left\langle \rho_{\sigma}^\lambda(w-k_2k_1^{-1}z,k_2k_1^{-1})
(\phi_\gamma^\lambda \otimes e_j^\sigma),\phi_\gamma^\lambda \otimes e_j^\sigma \right\rangle d w  ds.
\end{align*}
By (\ref{exp50}), the above expression takes the form
\begin{align*}
\sum\limits_{\substack{\gamma\in\mathbb{N}^n\\1\leq j\leq d_{\sigma}}}\sum_{|\alpha|=|\gamma|}\int_{K}
\eta_{\gamma \alpha}^\lambda(k_2k_1^{-1}) \int_{\mathbb{C}^n}g(w,k_2)e^{\frac{i\lambda}{2} \text{Im} (k_2k_1^{-1}z.\bar{w})}
\phi_{\alpha\gamma}^\lambda(w-k_2k_1^{-1}z)\varphi_{jj}^\sigma(k_2k_1^{-1})dw dk_2,
\end{align*}
where $\phi_{\alpha\gamma}^\lambda(x)=\langle\pi_\lambda(x)\phi_\alpha^\lambda,\phi_\gamma^\lambda\rangle.$
Then by the Peter-Weyl theorem (inversion) for the compact groups, we derive that
\begin{align}\label{exp53}
\sum\limits_{\sigma\in \hat{K}}d_\sigma \text{tr}(W_{\sigma}^\lambda(g)(\rho_{\sigma}^\lambda)^*(z,k_1))
&=\sum\limits_{\gamma\in \mathbb{N}^n} \sum_{|\alpha|=|\gamma|}
\eta_{\gamma \alpha}^\lambda(\mathsf{e}) \int_{\mathbb{C}^n}g(w,k_1)e^{\frac{i\lambda}{2} \text{Im} (z.\bar{w})}
\phi_{\alpha\gamma}^\lambda(w-z)dw  \nonumber \\
&=\sum\limits_{\gamma\in \mathbb{N}^n}\int_{\mathbb{C}^n}
g(w,k_1)e^{\frac{i\lambda}{2} \text{Im} (z.\bar{w})}\phi_{\gamma \gamma}^\lambda(w-z)dw, \nonumber
\end{align}
where $\mathsf{e}$ is the identity element in $K.$ Thus in view of the inversion formula for the Weyl transform on the
Heisenberg group, we infer that
\begin{align}
\sum\limits_{\sigma\in \hat{K}}d_\sigma \text{tr}(W_{\sigma}^\lambda(g)(\rho_{\sigma}^\lambda)^*(z,k_1))
=(2\pi)^n |\lambda|^{-n} g(z,k_1).
\end{align}
\end{proof}

For simplicity, we assume $\lambda=1$ and denote $\rho_\sigma(z,k)=\rho_\sigma^1(z,k),$ $W_{\sigma}=W_{\sigma}^{1}.$
Further, throughout this section, we shall assume $A$ is a Lebesgue measurable subset of
$\mathbb{C}^n$ with finite measure. Next, we define a set of orthogonal projection operators
which is core in formulating a problem analogous to Benedicks-Amrein-Berthier type theorem.

\smallskip

Let $\sigma\in\hat{K}$ and $\mathcal B_{N_\sigma}$ be an $N_\sigma$ dimensional subspace of
$\mathcal{H}_\sigma^2.$ Then, there exists an orthonormal basis $\{\psi_l^\sigma :l\in \mathbb{N}\}$
of $\mathcal{H}_\sigma^2$ such that $\mathcal B_{N_\sigma}=\text{span }\left\{\psi_l^{\sigma}: 1\leq l\leq N_\sigma \right\}.$
Define an orthogonal projection $P_{N_\sigma}$ of $\mathcal{H}_\sigma^2$ onto
$\mathcal{R}(P_{N_\sigma})=\mathcal B_{N_\sigma}.$ Consider a finite subset $J$ of $\hat{K}$
and let $N=\max\limits_{\sigma\in J}N_\sigma.$ Now, we define a pair of orthogonal projections $E_A$ and $F_N$ of
$L^2(G^\times)$ by
\[ E_A g =\chi_{A\times K}~g \quad \mbox{and} \quad
W_{\sigma}(F_N g)=\begin{cases}
                   P_{N_\sigma} W_\sigma(g) &\text{if }\sigma\in J,\\
                   0 &\text{otherwise,}
                 \end{cases}\]
where $\chi_{A\times K}$ denotes the characteristic function of $A\times K.$ Then, it is easy to see that
$\mathcal{R}(E_A)=\{g\in L^2(G^\times): g = g~\chi_{A\times K}\}$ and \[\mathcal{R}(F_N)
=\{g\in L^2(G^\times):\mathcal{R}(W_\sigma(g))\subseteq\mathcal B_{N_\sigma}\text{ for }\sigma\in J
\text{ and }\mathcal{R}(W_{\sigma}(g))=0\text{ for }\sigma\notin J\}.\]

Now, we derive a key lemma that enables us to recognize $E_AF_N$  as an integral operator.

\begin{lemma}\label{lemma51}
The operator $E_AF_N$ is an integral operator on $L^2(G^\times).$
\end{lemma}
\begin{proof}
Let $g\in L^2(G^\times)$. By inversion formula (\ref{exp55}) we get
\begin{align*}
(F_Ng)(z,k_1)&=\sum_{\sigma\in\hat{K}}a_\sigma\text{ tr}\left(W_\sigma(F_Ng) \rho_\sigma^*(z,k_1)\right)\\
&=\sum_{\sigma\in J}a_\sigma\text{ tr}\left(P_{N_\sigma}W_\sigma (g)\rho_\sigma^*(z,k_1)\right)\\
&=\sum_{\sigma\in J}a_\sigma\int_{K}\int_{\mathbb{C}^n} g(w,k_2) \text{ tr}\left(P_{N_\sigma} \rho_\sigma(w,k_2)
\rho_\sigma^*(z,k_1)\right)dw dk_2,
\end{align*}
where $a_\sigma=(2\pi)^{-n}d_\sigma.$
Hence,
\begin{align*}
(E_AF_Ng)(z,k_1)&=\chi_{A\times K}(z,k_1)(F_Ng)(z,k_1) \\
&=\int_{K}\int_{\mathbb{C}^n} g( w ,k_2)\mathcal{K}((z,k_1),( w ,k_2))d w  dk_2,
\end{align*}
where $\mathcal{K}\left((z,k_1),(w,k_2)\right)=\sum_{\sigma\in J}a_\sigma\chi_{A\times K}(z,k_1)\text{ tr}
\left(P_{N_\sigma} \rho_\sigma(w,k_2)\rho_\sigma^*(z,k_1)\right).$
\end{proof}

Further, the integral operator $E_AF_N$ is a Hilbert-Schmidt operator and
satisfies the following dimension condition.
\begin{lemma}\label{lemma52}
$E_AF_N$ is a Hilbert-Schmidt operator with $\|E_AF_N\|_{HS}^2\leq c_J m(A)N$,
where $c_J=(2\pi)^n m(K)|J|\sum_{\sigma\in J} a_\sigma^2<\infty.$
\end{lemma}
\begin{proof}
From Lemma \ref{lemma51} it follows that
\begin{align*}
\|E_AF_N\|_{HS}^2 & =\int_{G^\times}\int_{G^\times}|\mathcal{K}((z,k_1),( w ,k_2))|^2 dw dk_2 dz dk_1 \\
& =\int_{G^\times}\int_{G^\times}\left| \sum_{\sigma\in J}a_\sigma\chi_{A\times K}(z,k_1)\text{ tr}
\left(P_{N_\sigma} \rho_\sigma(w,k_2)\rho_\sigma^*(z,k_1)\right)\right|^2 dw dk_2 dz dk_1.
\end{align*}
If the cardinality of $J$ is denoted by $|J|,$ from H{\"o}lder's inequality, we get
\begin{align}\label{exp62}
\|E_AF_N\|_{HS}^2\leq |J| & \sum_{\sigma\in J} a_\sigma^2 \int_{G^\times}|\chi_{A\times K}(z,k_1)|^2 \\
& \int_{G^\times}|\text{tr}\left(P_{N_\sigma} \rho_\sigma(w,k_2)\rho_\sigma^*(z,k_1)\right)|^2 dw dk_2dz dk_1. \nonumber
\end{align}
Now, we shall simplify the inner integral
\begin{align*}
&\int_{G^\times}|\text{tr} \left(P_{N_\sigma} \rho_\sigma(w,k_2)\rho_\sigma^*(z,k_1)\right)|^2 dw dk_2\\
=&\int_{G^\times}\Big|\sum\limits_{1\leq l\leq N_\sigma} \langle\rho_\sigma (w,k_2)\rho_\sigma^*(z,k_1)
\psi_l^\sigma,\psi_l^\sigma\rangle \Big|^2 dw dk_2 \\
=&\int_{G^\times}\Big|\sum\limits_{1\leq l\leq N_\sigma} \langle\rho_\sigma(w,k_2)
\eta_l^\sigma,\psi_l^\sigma\rangle \Big|^2 dw dk_2,
\end{align*}
where $\eta_l^\sigma=\rho_\sigma^*(z,k_1)\psi_l^\sigma\in \mathcal{H}_\sigma^2.$
The above integral can be written in terms of Fourier-Wigner transform by
\begin{align*}
&\int_{G^\times}\Big| \sum\limits_{1\leq l\leq N_\sigma} \langle\rho_\sigma (w,k_2)
\eta_l^\sigma,\psi_l^\sigma\rangle \Big|^2 dw dk_2
=(2\pi)^n\int_{G^\times}\Big| \sum\limits_{1\leq l\leq N_\sigma}
V_{\eta_l^\sigma}^{\psi_l^\sigma}(w,k_2)\Big|^2 dw dk_2 \\
=& \,(2\pi)^n\sum\limits_{1\leq l_1,l_2\leq N_\sigma}\int_{G^\times} V_{\eta_{l_1}^\sigma}^{\psi_{l_1}^\sigma}(w,k_2)
\overline{V_{\eta_{l_2}^\sigma}^{\psi_{l_2}^\sigma}(w,k_2)}dw dk_2.
\end{align*}
Since, \[\langle\eta_{l_1}^\sigma,\eta_{l_2}^\sigma\rangle=\langle\rho_\sigma^*(z,k_1)
\psi_{l_1}^\sigma,\rho_\sigma^*(z,k_1)\psi_{l_2}^\sigma\rangle=\langle
\psi_{l_1}^\sigma,\psi_{l_2}^\sigma\rangle=\delta_{l_1l_2},\]
by Lemma \ref{lemma50}, we have
\begin{align}\label{exp64}
&\int_{G^\times}|\text{tr} \left(P_{N_\sigma} \rho_\sigma(w,k_2)\rho_\sigma^*(z,k_1)\right)|^2 dw dk_2 \nonumber \\
=& \,(2\pi)^n \sum\limits_{1\leq l_1,l_2\leq N_\sigma}\langle\eta_{l_1}^\sigma,\eta_{l_2}^\sigma\rangle
\overline{\langle\psi_{l_1}^\sigma,\psi_{l_2}^\sigma\rangle}
=(2\pi)^n N_\sigma.
\end{align}
Thus, from (\ref{exp62}) and (\ref{exp64}) we get
$$\|E_AF_N\|_{HS}^2 \leq (2\pi)^n m(A)m(K) N|J|\sum\limits_{\sigma\in J} a_\sigma^2<\infty,$$
where $N=\max\limits_{\sigma\in J}N_\sigma$ as defined above.
\end{proof}

We need the following result that describes an interesting property of Lebesgue
measurable sets \cite{AB}. Denote $ w  A=\{z\in \mathbb{C}^n:z- w  \in A\}.$
\begin{lemma}\cite{AB}\label{lemma53}
Let $B$ be a measurable set in $\mathbb{C}^n$ with $0<m(B)<\infty.$ If $B_0$
is a measurable subset of $B$ with $m(B_0)>0,$ then for each $\epsilon>0$ there exists
$ w \in \mathbb C^n$ such that
\[m(B)<m(B\cup w B_0)<m(B)+\epsilon.\]
\end{lemma}

We also need the following basic fact about the orthogonal projection,
which help in deciding the disjointness of the projections $E_A$ and $F_N$
while $m(A)<\infty.$

For given orthogonal projections $E$ and $F$ of a Hilbert
space $\mathcal H,$ let $E\cap F$ denote the orthogonal projection of $\mathcal H$
onto $\mathcal{R}(E)\cap \mathcal{R}(F).$ Then
\begin{align}\label{exp74}
\| E\cap F \|_{HS}^2=\dim\mathcal{R}( E\cap F)\leq \| EF \|_{HS}^2.
\end{align}
Let $F_{N}^{\perp}=I-F_N,$ and $A^c$ be the complement of $A.$

\begin{proposition}\label{prop52}
Let $A$ be a measurable subset of $\mathbb{C}^n$ of finite Lebesgue measure.
Then the projection $E_A\cap F_N=0$.
\end{proposition}
\begin{proof}
Assume towards a contradiction that there exists a non-zero function $g$ in
$\mathcal{R}( E_A\cap F_N)$. Then $\mathcal{R}(W_\sigma(g))\subseteq\mathcal B_{N_\sigma}$ for
$\sigma\in J$ and $\mathcal{R}(W_{\sigma}(g))=0$ for $\sigma\in \hat{K}\setminus J$. Consider
$A_0=\{z\in A: \exists \text{ a positive measure set } K_z \subseteq K \text{ with }g(z,k)\neq 0,  \forall k\in K_z \}$.
Then $0<m(A_0)<\infty$. Let $g_0(z,k)=\chi_{A_0}(z)g(z,k).$ Thus $g=g_0$ a.e. and hence $g_0\in \mathcal{R}( E_A\cap F_N).$
Choose $s\in \mathbb{N}$ such that $s>2c_J m(A_0)N.$ Now, we construct an increasing
sequence of sets $\{A_l:l=1,\ldots,s\}.$ Using Lemma \ref{lemma53} with $\epsilon=\frac{1}{2c_J N}$,
$B_0=A_0$ and $B=A_{l-1},$ there exists $w_l\in \mathbb{C}^n$ such that
\[m(A_{l-1})<m(A_{l-1}\cup w _lA_0)<m(A_{l-1})+\frac{1}{2c_J N}.\]
Denote $A_l=A_{l-1}\cup w _lA_0$. Then from (\ref{exp74}), we get
\begin{align}\label{exp56}
\dim\mathcal{R}( E_{A_s}\cap F_N) \leq c_J m(A_s)N < \left\{m(A_0)+\frac{s}{2c_J N}\right\}c_J N<s.
\end{align}
On the other hand, we construct $s+1$ linearly independent functions in the space
$\mathcal{R}( E_{A_s}\cap F_N),$ after verifying $\mathcal{R}(F_N)$ is a twisted
translation invariant space.

Let $g_l(z,k)=e^{\frac{i}{2}Im(z.\bar{ w _l})}g_0(z- w _l,k)$. Then for
$\eta^\sigma \in \mathcal{H}_\sigma^2$ and $p>N_\sigma,$ where $\sigma\in J,$ we have
\begin{align*}
\langle W_\sigma(g_l)\eta^\sigma,\psi_p^\sigma\rangle &=\int_{G^\times}
g_l(z,k) \langle \rho_\sigma(z,k)\eta^\sigma,\psi_p^\sigma\rangle dz dk \\
&=\int_{G^\times} e^{\frac{i}{2}Im(z.\bar{w_l})} g_0(z- w _l,k)
\langle \rho_\sigma(z,k)\eta^\sigma,\psi_p^\sigma\rangle dz dk \\
&=\int_{G^\times} e^{\frac{i}{2}Im(z.\bar{w _l})} g_0(z,k)
\langle \rho_\sigma(z+ w_l,k)\eta^\sigma,\psi_p^\sigma\rangle dz dk.
\end{align*}
Since $\rho_\sigma(z,k)\rho_\sigma(k^{-1} w,\mathsf{e})=e^{\frac{i}{2}Im(z.\bar{w})}
\rho_\sigma(z+w,k),$ where $\mathsf{e}$ is the identity element in $K,$ we get
\begin{align*}
\langle W_\sigma(g_l)\eta^\sigma,\psi_p^\sigma\rangle
&=\int_{G^\times} g_0(z,k)\langle \rho_\sigma(z,k)\rho_\sigma(k^{-1}w_l,\mathsf{e})\eta^\sigma,
\psi_p^\sigma\rangle dz dk \\
&=\int_{G^\times} g_0(z,k) \langle \rho_\sigma(z,k)\zeta^\sigma,\psi_p^\sigma\rangle dz dk \\
&=\langle W(g_0)\zeta^\sigma,\psi_p^\sigma\rangle=0.
\end{align*}
Thus, $\mathcal{R}(W_\sigma(g_l))\subseteq\mathcal B_{N_\sigma}$ for $\sigma\in J.$
Similarly, for $\sigma\notin J,$ it can be shown that $\mathcal{R}(W_{\sigma}(g_l))=0.$
Since $A_m=A_0\cup w _1A_0\cup\cdots\cup w_m A_0$ and $g_l=0$ on $(w_lA_0)^c\times K$, we
have $E_{A_m}g_l=g_l$ for $l=0,1,\ldots,m$. Furthermore, $E_{A_m \setminus A_{m-1}}g_l=0$
for $l=0,\ldots,m-1$ and by the definition of $A_0,$ it follows that $E_{A_m \setminus A_{m-1}}g_m\neq 0$
on a set of positive measure. Therefore, it shows that $g_m$ is not a linear combination of
$g_0,\ldots,g_{m-1}$. Hence, $g_0,\ldots,g_s$ are $s+1$ linearly independent functions in
$\mathcal{R}( E_{A_s}\cap F_N)$ that contradicts (\ref{exp56}). This completes the proof.
\end{proof}

This leads to the following version of the Benedicks-Amrein-Berthier theorem for the Weyl transform.
\begin{proposition}\label{prop58}
Let $g\in L^1(G^\times)$ and $\{(z,k)\in G^\times:g(z,k)\neq0\}\subseteq A\times K,$ where $m(A)<\infty.$
Suppose $J$ be a finite subset of $\hat{K}.$ If $W_{\sigma}(g)$ is a finite rank operator for each
$\sigma\in J$ and $W_{\sigma}(g)=0$ for $\sigma\in \hat{K}\setminus J,$ then $g=0.$
\end{proposition}

If $g\in L^1(G^\times),$ by the Plancherel theorem (\ref{exp51}), assumed rank condition implies
$g\in L^2(G^\times).$ Further, for $g\in L^2(G^\times),$ proof of Proposition \ref{prop58} follows
from Proposition \ref{prop52}.

In the Heisenberg motion group, in terms of Fourier transform, the above result takes the following form.
\begin{theorem}\label{th52}
Let $f\in L^1(G)$ and $\{(z,t,k)\in G:g(z,t,k)\neq0\}\subseteq A\times \mathbb{R}\times K,$ where
$m(A)<\infty.$ For each $\lambda\in\mathbb{R}^*,$ consider a finite subset $J_\lambda$ of
$\hat{K}.$ If for each $\lambda\in\mathbb{R}^*,$ $\hat{f}(\lambda,\sigma)$ has finite rank for
$\sigma\in J_\lambda$ and $\hat{f}(\lambda,\sigma)=0$ for $\sigma\in \hat{K}\setminus J_\lambda,$
then $f=0.$
\end{theorem}

\begin{remark}
Notice that, for the Heisenberg motion group, the QUP condition $(\ref{exp79})$ is equivalent to
\begin{equation}\label{exp81}
\int_{\mathbb{R}\setminus\{0\}}\left(\sum_{\sigma\in\hat K}d_\sigma\,\text{rank}\hat f(\lambda,\sigma)\right)
|\lambda|^nd\lambda <\infty.
\end{equation}
Therefore, the rank condition in Theorem \ref{th52} will not satisfy $(\ref{exp81}),$ and hence
Theorem \ref{th52} improves the QUP in that perspective. Further, the assumption in Theorem \ref{th52}
that, for each $\lambda\in\mathbb{R}^*,$ $\hat{f}(\lambda,\sigma)=0$ except finitely many $\sigma,$ looks
natural in view of $(\ref{exp81}).$
\end{remark}

As a consequence of Proposition \ref{prop58}, we obtain the following analogue result in terms
of the Fourier-Wigner decomposition. For this, we recall the Fourier-Wigner decomposition. Let
$g\in L^2(G^\times).$ By Proposition \ref{prop51}, we get $g=\bigoplus\limits_{\sigma\in\hat K}g_\sigma.$

\begin{proposition}\label{prop50}
Let $g\in L^2(G^\times)$ and $\{(z,k)\in G^\times:g_\sigma(z,k)\neq0\}\subseteq A_\sigma \times K,$
where $m(A_\sigma)<\infty,$ whenever $\sigma \in \hat{K}.$ If $W_{\sigma}(g)$ is a finite rank
operator for each $\sigma,$ then $g=0.$
\end{proposition}
\begin{proof}
For $\varphi, \psi\in\mathcal H_\sigma^2$ we have
\begin{align*}
\left\langle W_{\sigma}(g)\varphi,\psi\right\rangle &= \int_{K}\int_{\mathbb{C}^n} g(z,k)
\left\langle\rho_\sigma(z,k)\varphi,\psi\right\rangle dzdk \\
&= \int_{K}\int_{\mathbb{C}^n}{g_\sigma}(z,k)\left\langle\rho_\sigma(z,k)\varphi,\psi\right\rangle dzdk \\
&=\left\langle W_{\sigma}(g_{\sigma})\varphi,\psi\right\rangle.
\end{align*}
Hence, for $\sigma_o\in\hat{K},$ $\mathcal{R}(W_{\sigma_o}(g_{\sigma_o}))=\mathcal{R}(W_{\sigma_o}(g))$ be a finite
dimensional subspace of $\mathcal H_{\sigma_o}^2$ and $\mathcal{R}(W_{\sigma}(g_{\sigma_o}))=0$
for $\sigma(\neq \sigma_o)\in \hat{K}$. Thus by Proposition \ref{prop58} we get $g_{\sigma_o}=0.$
Since $\sigma_o \in \hat{K}$ is arbitrary, we infer that $g=0.$
\end{proof}

\begin{remark}
$\bf(a).$ For $U(n)$- bi invariant function, the rank condition in Proposition \ref{prop50} is obviously true.
Thus support condition is enough for the conclusion. In dimension one, it can argue by the fact that
each Fourier-Wigner piece will be of the form $\tilde{\bar{g}}_\alpha=\tilde{\bar{g}}_\alpha\times \Phi_{\alpha,\alpha},$
where $\tilde{\bar{g}}_\alpha(z)=\bar{g}_\alpha(-z),$ which is real analytic. Hence it can not have
finite support.

\smallskip

$\bf(b).$ After a close examination of the utility of $U(n)$ to obtain the
decomposition of $L^2(G^\times)$ as in Proposition \ref{prop51}, we observed that
$U(n)$-invariance is nevermore used except while realizing the irreducible action of
metaplectic repression $\mu_\lambda$ on $P_m^\lambda.$ If we consider a compact subgroup
$K$ of $U(n)$ which makes $(\mathbb H^n\rtimes K, K)$ a Gelfand pair, then $P_m^\lambda$
will be decomposed into finitely many irreducible pieces according to the metaplectic
representation $\mu_\lambda$ of $K.$ To avoid further complexity in the calculation, we have
preferred to prove the results for the Gelfand pair $(\mathbb H^n\rtimes U(n), U(n))$ instead
of $(\mathbb H^n\rtimes K, K).$
\end{remark}

\noindent\textbf{Strong annihilating pair:}
Let $A\subseteq \mathbb{R}$ and $\Sigma\subseteq \hat{\mathbb{R}}$ be measurable sets.
Then the pair $(A,\Sigma)$ is called {\em weak annihilating pair} if $\text{supp}\,f\subseteq A$
and $\text{supp}\,\hat{f}\subseteq \Sigma,$ implies $f=0.$ The pair $(A,\Sigma)$ is called
{\em strong annihilating pair} if there exists a positive number $C=C(A,\Sigma)$ such that
\begin{equation}\label{exp29}
\|f\|_2^2\leq C \left(\|f\|_{L^2(A^c)}^2+\|\hat{f}\|_{L^2(\Sigma^c)}^2 \right)
\end{equation}
for every $f \in L^2(\mathbb{R}).$ It is obvious that every strong annihilating
pair is a weak annihilating pair. In \cite{B}, Benedicks had proved that $(A,\Sigma)$ is a weak
annihilating pair when $A$ and $\Sigma$ both have finite measure. In \cite{AB}, Amrein-Berthier
had proved that $(A,\Sigma)$ is a strong annihilating under the identical assumption as in \cite{B}.

\smallskip

Since, Fourier transform on the Heisenberg motion group is an operator valued function, we
could not expect a similar conclusion as (\ref{exp29}). However, we can adequately describe
a strong annihilating pair.

\begin{definition}
For each $\sigma\in \hat{K},$ let $A_\sigma$ be a measurable subset of $\mathbb C^n$
and $S_\sigma$ be a closed subspace of $\mathcal H_\sigma^2.$ By abuse of notations,
denote $A=(A_\sigma)_{\sigma\in \hat{K}}$ and $S=(S_\sigma)_{\sigma\in\hat{K}}.$ We
say that the pair $(A,S)$ is a {\em strong annihilating pair} for the Weyl transform,
if there exist positive numbers $C_\sigma=C_\sigma(A_\sigma,S_\sigma)$ such that for
every $g \in L^2(\mathbb C^n \times K),$
\begin{equation}\label{exp60}
\|g\|_2^2\leq \sum_{\sigma \in \hat{K}} C_\sigma\left(\|g_\sigma \|_{L^2(A_\sigma^c\times K)}^2
+\|P_{S_{\sigma}}^{\perp} W_\sigma^\lambda(g) \|_{HS}^2 \right),
\end{equation}
where $g_\sigma$'s are the orthogonal pieces of $g,$ according to Proposition \ref{prop51},
and $P_{S_{\sigma}}$ is the projection of $\mathcal H_\sigma^2$ onto $S_\sigma.$
\end{definition}

We prove that if $A_\sigma$ has finite measure and dimension of $S_\sigma$ is finite
for each $\sigma\in\hat{K},$ then $(A,S)$ is a strong annihilating pair. For this,
we need the following basic result.
\begin{lemma}\cite{HJ}\label{lemma7}
Let $P$ and $Q$ be two orthogonal projections on a complex Hilbert space $H.$
Then $\|PQ\|<1$ if and only if there exists a positive number $C$ such that for each $x\in H$
\[\|x\|^2\leq C\left(\|P^{\perp}x\|^2+\|Q^{\perp}x\|^2\right).\]
\end{lemma}

For $\sigma_o\in\hat{K},$ let $S_{\sigma_o}$ be a finite dimensional subspace of $\mathcal H_{\sigma_o}^2$
and $A_{\sigma_o}$ be any subset of $\mathbb{C}^n$ with finite measure. Then, recall the set of projections
for $J=\{\sigma_0\},$ $E_{A_{\sigma_o}}g(z,k)=\chi_{A_{\sigma_o}}(z)g(z,k)$ and
$W_{\sigma_o}^\lambda (F_{S_{\sigma_o}}g)=P_{S_{\sigma_o}}W_{\sigma_o}^\lambda(g),$ $W_\sigma^\lambda (F_{S_{\sigma_o}}g)=0$
for $\sigma\neq \sigma_o.$ Now, by Lemma \ref{lemma52} $E_{A_{\sigma_o}}F_{S_{\sigma_o}}$ is a compact
operator and from Proposition \ref{prop52}, we have $E_{A_{\sigma_o}}\cap F_{S_{\sigma_o}}=0.$ Therefore,
we must have $\|E_{A_{\sigma_o}}F_{S_{\sigma_o}}\| <1.$ Since
$W_{\sigma_o}^\lambda (F_{S_{\sigma_o}}^\perp g)=P_{S_{\sigma_o}}^{\perp}W_{\sigma_o}^\lambda(g)$
and $W_\sigma^\lambda (F_{S_{\sigma_o}}^\perp g)=W_\sigma^\lambda(g)$ for $\sigma\neq \sigma_o,$ by
Lemma \ref{lemma7}, there exists $\tilde{C}_{\sigma_o}=\tilde{C}_{\sigma_o}(A_{\sigma_o},S_{\sigma_o})>0$ such that
\[\|g\|_2^2\leq \tilde{C}_{\sigma_o}\left(\|g\|_{L^2(A_{\sigma_o}^c\times K)}^2
+d_{\sigma_o}\|P_{S_{\sigma_o}}^{\perp}W_{\sigma_o}^\lambda(g)\|_{HS}^2
+\sum_{\sigma\neq\sigma_o}d_\sigma\|W_\sigma^\lambda(g)\|_{HS}^2\right),\]
for all $g \in L^2(\mathbb C^n \times K).$ In particular, for any $g_{\sigma_o}\in V_{\sigma_o}^\lambda$ we have
\begin{equation}\label{exp59}
\|g_{\sigma_o}\|_2^2\leq C_{\sigma_o}\left(\|g_{\sigma_o}\|_{L^2(A_{\sigma_o}^c\times K)}^2
+\|P_{S_{\sigma_o}}^{\perp}W_{\sigma_o}^\lambda(g_{\sigma_o})\|_{HS}^2 \right).
\end{equation}
For any $g\in L^2(\mathbb C^n \times K),$ by Proposition \ref{prop51}, $g=\bigoplus\limits_{\sigma\in\hat K}g_\sigma.$
Since $\sigma_o$ is arbitrary, from (\ref{exp59}) we can conclude that $(A,S)$ is a strong annihilating pair, whenever
$A_\sigma$ has finite measure and dimension of $S_\sigma$ is finite for each $\sigma\in\hat{K}.$

\begin{remark}
Consider the hypothesis of Proposition \ref{prop58}. There exist two large classes of functions
of which one satisfies the support condition, and the other satisfies the rank condition.
However, in Proposition \ref{prop50}, it is not clear which functions will fulfill such a
support condition. In other words, whether the assumption of finite support condition in each
piece is strong enough for the conclusion of Proposition \ref{prop50}. We know this is true for
the $U(n)$-bi-invariant functions. However, we reached out to a quantitative estimate
$(\ref{exp60},\,\ref{exp59})$ of Proposition \ref{prop50}, which is true for all square
integrable functions, irrespective of their support.
\end{remark}

\section{Quaternion Heisenberg group}
In this section, we prove an analogue of Benedicks-Amrein-Berthier theorem for
the quaternion Heisenberg group. As the multiplication of two quaternions is not
commutative, the generalization of the above result in this setup is notable. We
shall start by describing notations and preliminary facts about the quaternion
Heisenberg group. For more details, see \cite{CM,CZ}.

\smallskip

The quaternion Heisenberg group is step two nilpotent
Lie group with centre $\mathbb R^3.$ Let $\mathbb{Q}$ be the set of all quaternions. For $q=q_0+iq_1+jq_2+kq_3 \in \mathbb{Q}$,
the conjugate of $q$ is defined by $\bar{q}=q_0-iq_1-jq_2-kq_3.$ The inner product in $\mathbb{Q}$
is defined by $\langle q,\tilde{q}\rangle = \text{Re}(\bar{q}\tilde{q}).$ This leads to
$|q|^2=\langle q,q\rangle=\sum\limits_{l=0}^3 q_l^2,$ and we get the
relations $\overline{q\tilde{q}}=\bar{\tilde{q}}\bar{q}$ and $|q\tilde{q}|=|q\|\tilde{q}|.$
The set $\mathcal{Q}=\mathbb{Q}\times \mathbb{R}^3=\{(q,t):q\in \mathbb{Q}, t\in \mathbb{R}^3\}$
becomes a non-commutative group when equipped with the group law
\[(q,t)(\tilde{q},\tilde{t})=(q+\tilde{q},t+\tilde{t}-2~\text{Im}(\bar{\tilde{q}}q)).\]
It is easy to see that the Lebesgue measure $dq dt$ on $\mathbb{Q}\times \mathbb{R}^3$
is the Haar measure on $\mathcal{Q}.$ For $1\leq p \leq \infty,$ $L^p(\mathcal{Q})$
denotes the usual $L^p$ space of all quaternion-valued functions on $\mathbb{Q}\times \mathbb{R}^3.$

\smallskip

Let $a\in \text{Im }\mathbb{Q}\smallsetminus \{0\}.$ Then $J_a:q\mapsto q\cdot\frac{a}{|a|}$
defines a complex structure on $\mathbb{Q}.$ Let $\mathcal{F}_a$ be the Fock space of all holomorphic
functions $F$ with quaternion values on $(\mathbb{Q},J_a)$ such that
\[\| F \|_2^2=\int_{\mathbb{Q}}|F(q)|^2 e^{-2|a\|q|^2}dq <\infty. \]
An irreducible unitary representation $\pi_a$ of $\mathcal{Q}$ realized on $\mathcal{F}_a$
is given by \[\pi_a(q,t) F(\tilde{q})=e^{i\langle a,t\rangle-|a|\left(|q|^2+2\langle \tilde{q},q\rangle
-2i\langle \tilde{q}.\frac{a}{|a|},q\rangle\right)}F(\tilde{q}+q),\]
where $F\in \mathcal{F}_a.$ Up to unitary equivalence, $\pi_a$'s are  all the infinite
dimensional irreducible unitary representations of $\mathcal{Q}.$ For $f\in L^1(\mathcal{Q}),$
the group Fourier transform can be expressed as
\[\hat{f}(a)=\int_{\mathcal{Q}}f(q,t)\pi_a(q,t)dqdt.\]
For $f\in L^2(\mathcal{Q}),$ the following Plancherel formula holds.

\begin{equation}\label{exp80}
\| f\|_2^2=\frac{1}{2\pi^5}\int_{\text{Im } \mathbb{Q}\setminus \{0\}}\|\hat{f}(a)\|_{\text{HS}}^2|a|^2da.
\end{equation}

Let \[f^a(q)=\int_{\mathbb{R}^3}f(q,t)e^{i\langle a,t\rangle} dt,\] the inverse Fourier transform of $f$ in the $t$ variable.
If we denote $\pi_a(q)=\pi_a(q,0),$ then the Weyl transform of $g\in L^1(\mathbb{Q})$ can be define by
\begin{align*}
W_a(g)=\int_{\mathbb{Q}}g(q)\pi_a(q)dq.
\end{align*}
Hence it follows that $\hat{f}(a)=W_a(f^a).$ Further, $W_a(g)$ is a bounded operator if $g\in L^1(\mathbb{Q})$
and a Hilbert-Schmidt operator when $g\in L^2(\mathbb{Q}).$ In addition, when $g\in L^2(\mathbb{Q}),$ the
Plancherel formula for the Weyl transform $W_a$ is given by
\begin{align}\label{exp70}
\| W_a(g)\|_{\text{HS}}^2=\frac{\pi^2}{4|a|^2}\|g\|^2.
\end{align}
Finally, the inversion formula for the Weyl transform is given by
\begin{align}\label{exp71}
g(q)=\frac{4|a|^2}{\pi^2}\text{tr}\left(W_a(g)\pi_a^*(q)\right).
\end{align}

\smallskip

In order to prove that $\pi_a$ is a square integrable representation, we need to recall
the following Schur's orthogonality relation, see \cite{MW}.

\begin{proposition}\label{prop72}\cite{MW}
Suppose $G$ be a connected simply connected nilpotent Lie group with centre $Z.$ Let $\pi$ be
an irreducible unitary representation of $G$ realised on a complex Hilbert space $H$ and
$\pi|_Z=\chi I_H,$ where $\chi$ is a character of $Z.$ Then $\langle\pi(x)h,h\rangle\in L^2(G/Z)$
for some non-zero $h\in H$ if and only if
\begin{align*}
\int_{G/Z}\langle\pi(x)h_1,k_1\rangle\overline{\langle\pi(x)h_2,k_2\rangle}d\nu_{G/Z}
=c_\pi\langle h_1,h_2\rangle\overline{\langle k_1,k_2\rangle}
\end{align*}
for all $h_l,k_l\in H,l=1,2,$ where $c_\pi$ is a constant depending only on $\pi.$
\end{proposition}

\begin{lemma}\label{lemma70}
Let $\varphi,\psi \in \mathcal{F}_a.$ Then,
\begin{align}\label{exp72}
\int_{\mathbb{Q}}|\langle \pi_a(q)\varphi,\psi\rangle|^2 dq \leq 4c_a\|\varphi\|_2^2\|\psi\|_2^2,
\end{align}
where $c_a$ is some constant.
\end{lemma}

\begin{proof}
For $q=q_0+iq_1+jq_2+kq_3 \in \mathbb{Q},$ write $z_1=q_0+iq_1$ and $z_2=q_3+iq_4.$
Then $\varphi(q)=\frac{4|a|^2}{\pi}z_1z_2\in \mathcal{F}_a$ and $\|\varphi\|_2=1.$
Further,
\begin{align*}
\int_{\mathbb{Q}}|\langle \pi_a(q)\varphi,\varphi \rangle|^2 dq&=\int_{\mathbb{Q}}\Big|\int_{\mathbb{Q}}
\pi_a(q)\varphi(\tilde{q})\overline{\varphi(\tilde{q})}e^{-2|a||\tilde{q}|^2}d\tilde{q}\Big|^2 dq \\
&=\int_{\mathbb{Q}}\Big|\int_{\mathbb{Q}}e^{-|a|(|q|^2+2\langle \tilde{q},q\rangle-2i\langle \tilde{q}
\frac{a}{|a|},q\rangle}\varphi(\tilde{q}+q)\overline{\varphi(\tilde{q})}e^{-2|a||\tilde{q}|^2}d\tilde{q}\Big|^2 dq.
\end{align*}
By Minkowski's integral inequality, it follows that
\begin{align*}
\left(\int_{\mathbb{Q}}|\langle \pi_a(q)\varphi,\varphi\rangle|^2 dq\right)^\frac{1}{2}
&\leq\int_{\mathbb{Q}}\left(\int_{\mathbb{Q}}|\varphi(\tilde{q}+q)|^2e^{-2|a||\tilde{q}+q|^2}
 |\varphi(\tilde{q})|^2e^{-2|a||\tilde{q}|^2}dq\right)^\frac{1}{2}d\tilde{q} \\
&=\|\varphi\|_2\int_{\mathbb{Q}}|\varphi(\tilde{q})|e^{-|a||\tilde{q}|^2}d\tilde{q} \\
&=\frac{4|a|^2}{\pi}\int_{\mathbb{C}^2}|z_1\|z_2|e^{-|a|(|z_1|^2+|z_2|^2)}dz_1 dz_2 <\infty.
\end{align*}
Hence by Proposition \ref{prop72}, for complex valued functions $\varphi_l,\psi_l \in \mathcal{F}_a,$
where $l=1,2,$ we get
\begin{align}\label{exp76}
\int_{\mathbb{Q}}\langle \pi_a(q)\varphi_1,\psi_1\rangle\overline{\langle \pi_a(q)\varphi_2,\psi_2}\rangle dq
=c_a\langle \varphi_1,\varphi_2 \rangle \overline{\langle \psi_1,\psi_2 \rangle}.
\end{align}
For arbitrary $\varphi,\psi \in \mathcal{F}_a,$ we can write $\varphi=\varphi_1+\varphi_2 j$
and $\psi=\psi_1+\psi_2 j,$ where $\varphi_l,\psi_l$ are complex valued functions. Then by (\ref{exp76}),
\[\int_{\mathbb{Q}}|\langle \pi_a(q)\varphi,\psi\rangle|^2 dq
\leq 4c_a \sum_{l_1,l_2=1}^2\|\varphi_{l_1}\|_2^2\|\psi_{l_2}\|_2^2=4c_a\|\varphi\|_2^2\|\psi\|_2^2,\]
where the last equality true as $|\varphi(\tilde{q})|^2=|\varphi_1(\tilde{q})|^2+|\varphi_2(\tilde{q})|^2.$
\end{proof}

Let $g\in L^2(\mathbb{Q})$ and $W_a(g)$ be a finite rank operator. Then there
exists an orthonormal basis, say, $\{e_1,e_2,\ldots\}$ of $\mathcal{F}_a$ such that
$\mathcal{R}(W_a(g))=\mathcal{B}_N,$ where $\mathcal{B}_N=\text{span}\{e_1,\ldots,e_N \}.$
Define an orthogonal projection $P_N$ of $\mathcal{F}_a$ onto $\mathcal{B}_N.$
Let $A$ be a measurable subset of $\mathbb{Q}$. Define a pair of orthogonal
projections $E_A$ and $F_N$ of $L^2(\mathbb{Q})$ by
\begin{align}\label{exp77}
E_A g=\chi_A g \qquad \text{ and } \qquad W_a(F_N g)=P_N W_a(g),
\end{align}
where $\chi_A$ denotes the characteristic function of $A.$
Then $\mathcal{R}(E_A)=\{g\in L^2(\mathbb{Q}): g= \chi_A g \}$ and
$\mathcal{R}(F_N)=\{g\in L^2(\mathbb{Q}): \mathcal{R}(W_a(g))\subseteq\mathcal{B}_N\}.$

\smallskip

Next, we prove that $E_AF_N$ is a Hilbert-Schmidt operator. Throughout this section,
we shall assume that $A$ is a measurable subset of $\mathbb{Q}$ with finite measure.

\begin{lemma}\label{lemma71}
$E_AF_N$ is an integral operator on $L^2(\mathbb{Q}).$
\end{lemma}
\begin{proof}
For $g\in L^2(\mathbb{Q})$, we have $W_a(F_N g)=P_N W_a(g)$. By inversion formula for the Weyl transform
\begin{align*}
(F_Ng)(q)&=\tilde{c}_a\text{tr}(W_a(F_Ng)\pi_a(q)^{*})=\tilde{c}_a\text{tr}(P_NW_a(g)\pi_a(-q)) \\
&=\tilde{c}_a\int_{\mathbb{Q}} g(\tilde{q})\text{tr}\left(P_N \pi_a(\tilde{q})\pi_a(-q)\right)d\tilde{q}.
\end{align*}
Hence, it follows that
\begin{align*}
(E_AF_Ng)(q)&=\chi_A(q)(F_Ng)(q) =\tilde{c}_a\chi_A(q)\int_{\mathbb{Q}} g(\tilde{q})\text{tr}
\left(P_N \pi_a(\tilde{q})\pi_a(-q)\right)d\tilde{q} \\
&=\tilde{c}_a\int_{\mathbb{Q}} g( w )K(q,\tilde{q})d\tilde{q},
\end{align*}
where $K(q,\tilde{q})= \tilde{c}_a~\chi_A(q)\text{tr}\left(P_N \pi_a(\tilde{q}) \pi_a(-q) \right)$ and
$\pi^2\tilde{c}_a=4|a|^2$. Thus, we infer that $E_AF_N$ is an integral operator with kernel $K.$
\end{proof}

\begin{lemma}\label{lemma72}
$E_AF_N$ is a Hilbert-Schmidt operator and $\|E_AF_N\|_{HS}^2\leq c'm(A)N^2$, for some constant $c',$
independent of the choice of $A$ and $N$.
\end{lemma}
\begin{proof}
From Lemma \ref{lemma71}, it follows that
\begin{align*}\label{exp25}
\|E_AF_N\|_{HS}^2&=\tilde{c}_a^2 \int_{\mathbb{Q}}\int_{\mathbb{Q}}|K(q,\tilde{q})|^2 d\tilde{q} dq   \nonumber \\
&=\tilde{c}_a^2 \int_{\mathbb{Q}}|\chi_A(q)|^2 \left(\int_{\mathbb{Q}}
|\text{tr}\,\left(P_N \pi_a(\tilde{q})\pi_a(-q)\right)|^2 d\tilde{q}\right)dq  \nonumber \\
&=\tilde{c}_a^2 \int_{\mathbb{Q}}\chi_A(q) \left(\int_{\mathbb{Q}}\Big|
\sum\limits_{l=1}^N \langle\pi_a(\tilde{q})\pi(-q)e_l,e_l\rangle \Big|^2 d\tilde{q}\right)dq   \nonumber  \\
&=\tilde{c}_a^2 \int_{\mathbb{Q}}\chi_A(q) \left(\int_{\mathbb{Q}}\Big|\sum\limits_{l=1}^N
e^{-2i\langle \tilde{q}a,q\rangle}\langle\pi_a(\tilde{q}-q)e_l,e_l\rangle \Big|^2 d\tilde{q}\right)dq.
\end{align*}
Since $\pi_a(\tilde{q})\pi_a(q)=e^{2i\langle \tilde{q}a,q\rangle}\pi_a(\tilde{q}+q)),$ we get
\begin{align*}
\|E_AF_N\|_{HS}^2\leq \tilde{c}_a^2 \int_{\mathbb{Q}}\chi_A(q) \left(N\int_{\mathbb{Q}}\sum\limits_{l=1}^N
\Big|\langle\pi_a(\tilde{q})e_l,e_l\rangle \Big|^2 d\tilde{q}\right)dq.
\end{align*}
Hence, from the square integrable property (\ref{exp72}) of the representation, it follows that
\begin{align*}
\|E_AF_N\|_{HS}^2\leq 4c_a\tilde{c}_a^2 m(A)N^2=c'm(A)N^2<\infty.
\end{align*}
\end{proof}
Let $S$ be a closed subspace of $\mathcal{F}_a.$
Define $F_S$ by $W_a(F_Sg)=P_SW_a(g),$ where $P_S$ is the orthogonal projection
of $\mathcal{F}_a$ onto $S$ and $g\in L^2(\mathbb{Q}).$ In particular,
if $S=\mathcal{B}_N,$ then $F_S=F_N.$

\begin{proposition}\label{prop74}
Let $A \subseteq \mathbb{Q}$ with finite measure, and $S$ be a closed subspace of $\mathcal{F}_a.$
Then, either $E_A\cap F_S=0$ or for each $\epsilon'>0$ there exists $\tilde{A}\supset A$ with
$m(\tilde{A}\smallsetminus A)<\epsilon'$ such that $\dim\mathcal{R}(E_{\tilde{A}}\cap F_S)=\infty.$
\end{proposition}

\begin{proof}
If $E_A\cap F_S\neq 0,$ then there exists a non-zero function $g_0\in\mathcal{R}( E_A\cap F_S)$.
Let $A_0=\{x\in A: g_0(x)\neq 0\}$ and $\tilde{A}_1=A.$ By Lemma \ref{lemma53}, for
$\epsilon=\frac{\epsilon'}{2^l}$, $B_0=A_0$ and $B=\tilde{A}_l,$ there exists $\tilde{q}_l\in \mathbb{Q}$
such that
\begin{equation}\label{exp78}
m(\tilde{A}_l)<m(\tilde{A}_l\cup \tilde{q}_lA_0)<m(\tilde{A}_l)+\frac{\epsilon'}{2^l}
\end{equation}
where $l\in \mathbb{N}.$
Put $\tilde{A}_{l+1}=\tilde{A}_l\cup \tilde{q}_lA_0$ and $\tilde A=\bigcup\limits_{l=1}^\infty\tilde{A}_l.$
Then $\tilde A_l$ is a non-decreasing sequence, and hence from (\ref{exp78}) it follows that
$m(\tilde{A}\smallsetminus A)<\epsilon'.$ For $l\in \mathbb{N},$ define
$g_l(q)=g_0(q-\tilde{q}_l)e^{2i\langle qa,\tilde{q}_l\rangle}.$  We show
that $g_l\in \mathcal{R}( E_{\tilde A}\cap F_S)$ for each $l$ and they are
linearly independent. Let $\mathcal{B}_S$ be an orthonormal basis of $S.$ Then
we can extend  $\mathcal{B}_S$ to an orthonormal basis  $\mathcal{B}$ of $\mathcal{F}_a.$
For $\varphi \in \mathcal{F}_a$ and $e_\beta\in \mathcal{B}\smallsetminus\mathcal{B}_S,$ we have
\begin{align*}
\langle W_a(g_l)\varphi,e_\beta\rangle &=\int_{\mathbb{Q}}
g_l(q) \langle \pi_a(q)\varphi,e_\beta\rangle dq\\
&=\int_{\mathbb{Q}} g_0(q-\tilde{q}_l)e^{2i\langle qa,\tilde{q}_l\rangle}
\langle \pi_a(q)\varphi,e_\beta\rangle dq.
\end{align*}
Since $\langle (q+\tilde{q}_l)a,\tilde{q}_l\rangle
=\langle q a,\tilde{q}_l\rangle,$ by the change of variables
\begin{align*}
\langle W_a(g_l)\varphi,e_\beta\rangle
&=\int_{\mathbb{Q}} g_0(q)e^{2i\langle qa,\tilde{q}_l\rangle}
\langle \pi_a(q+\tilde{q}_l)\varphi,e_\beta\rangle dq.
\end{align*}
We know that $\pi_a(q)\pi_a(\tilde{q}_l)=e^{2i\langle qa,\tilde{q}_l)\rangle}\pi_a(q+\tilde{q}_l).$ Therefore,
\begin{align}\label{exp75}
\langle W_a(g_l)\varphi,e_j\rangle
&=\int_{\mathbb{Q}} g_0(q)\langle \pi_a(q)\pi_a(\tilde{q}_l)\varphi,e_\beta\rangle dq \nonumber \\
&=\int_{\mathbb{Q}} g_0(q) \langle \pi_a(q)\psi,e_\beta\rangle dq \nonumber \\
&=\langle W_a(g_0)\psi,e_\beta\rangle=0.
\end{align}
Hence $\mathcal{R}(W_a(g_l))\subseteq S.$ Let $A_l=A_{l-1}\cup \tilde{q}_lA_0.$
Then $\tilde{A}_{l+1}=\tilde{A}_l\cup A_l.$ Thus, we have
$m(A_l\smallsetminus A_{l-1})\geq m(\tilde{A}_{l+1}\smallsetminus \tilde{A}_l)>0.$
Let $s\in \mathbb{N}.$ Since, $A_s=A_0\cup \tilde{q}_1A_0\cup\cdots\cup \tilde{q}_s A_0$ and $g_l=0$ on $(\tilde{q}_lA_0)^c,$
we have $E_{A_s}g_l=g_l$ for $l=0,1,\ldots,s$. Furthermore, $E_{A_s \smallsetminus A_{s-1}}g_l=0$ for
$l=0,\ldots,s-1$ and $E_{A_s \smallsetminus A_{s-1}}g_s\neq 0$. Therefore, it shows that $g_s$ is not a
linear combination of $g_0,\ldots,g_{s-1}.$ Since $s$ is arbitrary, $\{g_l:l\in \mathbb{N}\}$
is a linearly independent set in $\mathcal{R}( E_{\tilde A}\cap F_S).$
\end{proof}

\begin{proposition}\label{prop70}
Let $A$ be a measurable subset of $\mathbb{Q}$ having finite measure. Then the projection $E_A\cap F_N=0$.
\end{proposition}
\begin{proof}
In view of (\ref{exp74}) and Lemma \ref{lemma72}, we obtain the relations
\[\dim\mathcal{R}( E_{\tilde{A}}\cap F_N) \leq c'm(\tilde{A})N^2 <\infty. \]
Therefore, as a corollary of Proposition \ref{prop74}, we get $E_A\cap F_N=0.$
\end{proof}

The following theorem is our main result of this section, which is an analogue
of Benedicks-Amrein-Berthier theorem on the quaternion Heisenberg group.

\begin{theorem}\label{th70}
Let $A\subseteq\mathbb Q$ be a set of finite measure. Suppose $f\in L^1(\mathcal{Q})$
and $\{(q,t)\in \mathcal{Q}:f(q,t)\neq 0\} \subseteq A\times \mathbb{R}^3.$ If $\hat{f}(a)$
is a finite rank operator for each $a \in \text{Im } \mathbb{Q}\setminus \{0\}$, then $f=0$.
\end{theorem}

Proof of Theorem \ref{th70} follows from the following result for the Weyl transform
on $\mathbb Q.$

\begin{proposition}\label{prop71}
Let $g\in L^1(\mathbb{Q})$ and $\{q\in \mathbb{Q}:g(q)\neq 0\} \subseteq A,$ where $m(A)$ is finite.
If there exists $a \in \text{Im } \mathbb{Q}\setminus \{0\}$ such that $W_a(g)$ has finite rank,
then $g=0$.
\end{proposition}
Since $W_a(g)$ is a finite rank operator, by the Plancherel theorem for the Weyl
transform on $\mathbb Q,$ it follows that $g\in L^2(\mathbb{Q})$. Hence, it is enough to prove Proposition \ref{prop71}
for $g\in L^2(\mathbb{Q}).$ The prove of Proposition \ref{prop71} follows from Proposition \ref{prop70}.

\begin{remark}
If $0<m(A)<\infty,$ then $\dim\mathcal{R}(E_A)=\infty.$ In view of Proposition \ref{prop70}
and the fact that $E_A=(E_A\cap F_N) +(E_A\cap F_{N}^{\perp})=(E_A\cap F_{N}^{\perp}),$
it follows that $\dim\mathcal{R}(E_A\cap F_{N}^{\perp})=\infty.$ Since $m(A^c)=\infty,$
there exists a measurable set $B\subseteq A^c$ satisfying $0<m(B)<\infty.$ Hence
$\mathcal{R}(E_{A^c}\cap F_{N}^{\perp})\supseteq \mathcal{R}(E_B\cap F_{N}^{\perp})$.
This implies $\dim\mathcal{R}(E_{A^c}\cap F_{N}^{\perp})=\infty$. Similarly it can be shown
that $\dim\mathcal{R}(E_{A^c}\cap F_N)=\infty$.
\end{remark}

As a quantitative version of Proposition \ref{prop71}, we can obtain the following result, which
will follow from Lemma \ref{lemma7}, Lemma \ref{lemma72} and Proposition \ref{prop70}.
\begin{proposition}
Let $A$ be a measurable subset of $\mathbb{Q}$ with finite measure and $S$ be a finite
dimensional subspace of $\mathcal{F}_a$. Then, there exists a positive number $C=C(A,S)$
such that for every $g \in L^2(\mathbb{Q}),$
\begin{align*}
\|g\|_2^2\leq C \left(\|g\|_{L^2(A^c)}^2+\|P_S^{\perp} W_a(g) \|_{HS}^2\right),
\end{align*}
where $P_S$ is the projection of $\mathcal{F}_a$ onto $S.$
\end{proposition}

\smallskip

\noindent{\bf Acknowledgements:}
The authors wish to thank E. K. Narayanan for several fruitful discussions during the preparation 
of the manuscript. The first author gratefully acknowledges the support provided by IIT Guwahati, 
Government of India.


\end{document}